\documentclass{amsart}
\usepackage[colorlinks=true,linkcolor=blue, citecolor=blue,
                urlcolor=blue, pagebackref=false]{hyperref}

\makeatletter

\newcounter{maintheorem}[equation]
\def\themaintheorem{\thesection.\@arabic \c@maintheorem}
\@addtoreset{equation}{section}
\def\theequation{\thesection.\@arabic \c@equation}
\def\theenumi{\@alph\c@enumi}
\def\theenumii{\@roman\c@enumii}
\makeatother

\usepackage{amssymb,amsmath}
\usepackage{tabularray}
\usepackage{arydshln} %dash line
\usepackage{caption}
\usepackage[table,xcdraw]{xcolor} % To use color in tables

\usepackage{multirow} %for table
\usepackage{booktabs}
\usepackage{amsthm} 
\usepackage{mathabx} %required for widecheck
\usepackage{enumerate}
\usepackage{geometry}
\usepackage{stmaryrd} %power series brackets\usepackage{multirow}
\usepackage{relsize}    % for \mathlarger to work
\usepackage{graphicx} 
\usepackage{mathtools} %required for \xleftrightrrow{}
\usepackage{url}
\newtheorem{theorem}{Theorem}[section]
\newtheorem{lemma}[theorem]{Lemma}
\newtheorem{proposition}[theorem]{Proposition}
\newtheorem{corollary}[theorem]{Corollary}
\newtheorem{conjecture}[theorem]{Conjecture}

\newtheorem{question}[theorem]{Question}

\newtheorem{innercustomthm}{Theorem}
\newenvironment{customthm}[1]
  {\renewcommand\theinnercustomthm{#1}\innercustomthm}
  {\endinnercustomthm}

\theoremstyle{definition}
\newtheorem{definition}[theorem]{Definition}

\newtheorem{remark}[theorem]{Remark}
\newtheorem{example}[theorem]{Example}

\numberwithin{equation}{section}

\DeclareMathOperator{\im}{im}
\DeclareMathOperator{\HF}{HF}

\DeclareMathOperator{\Soc}{Soc}
\newcommand{\ann}{\mbox{\rm Ann}}                 
\newcommand{\Hom}{\mbox{\rm Hom}}

\newcommand{\m}{{\mathfrak m }}

\newcommand{\h}{\rm H}
\everymath{\displaystyle}

\title[ Symmetric decomposition of the Hilbert function of an ideal ]{Symmetric decomposition of the Hilbert function of an ideal}

\author[Bhat]{Meghana Bhat}
\address[Meghana Bhat]{Department of Mathematics, Indian Institute of Technology Dharwad, Permanent Campus, Chikkamalligwad,  Dharwad - 580011, Karnataka, India}
\email{mbhat.math@gmail.com}

\author[Dubey]{Saipriya Dubey}
\address[Saipriya Dubey]{Chennai Mathematical Institute, Plot H1 SIPCOT IT Park, Siruseri, Kelambakkam - 603103, Tamil
Nadu, India}
\email{saipriya721@gmail.com}

\author[Masuti]{Shreedevi K. Masuti}
\address[Shreedevi K. Masuti]{Department of Mathematics, Indian Institute of Technology Dharwad, Permanent Campus, Chikkamalligwad,  Dharwad - 580011, Karnataka, India}
\email{shreedevi@iitdh.ac.in}

\thanks{MB is supported by Prime Minister's Research Fellowship (PMRF), Govt. of India. 
SD is partially supported by a grant from the Infosys Foundation. SKM is supported by
CRG grant CRG/2022/007572 and MATRICS grant MTR/2022/000816  funded by SERB,  Govt. of India. The project is partially supported by SPARC grant SPARC/2019-2020/P1566/SL}

 %; Secondary:  13H10}

\date{\today}

\begin{document}
%%%%%%%%%%%%%%%%%%%%%%%%%%%%%%%%%%%%

%%%%%%%%%%%%%%%%%%%%%%%%%%%%%%%%%%%%
%%% Abstract
\begin{abstract}
Let $(R, \mathcal{M})$ be a local ring over a field $k$ with $k = R/\mathcal M$ and $J$ an ideal in $R$ 
such that $A =R/J$ is an Artinian Gorenstein (AG) $k$-algebra. In \cite{I89} A. Iarrobino introduced the symmetric decomposition of the Hilbert function of $A$. This became a very powerful tool for classifying the Hilbert functions of AG $k$-algebras. 
In this article, we introduce the symmetric decomposition of the Hilbert function of any ideal $I$ in $A.$  Our hope is that this result will be useful in classifying the possible Hilbert function of an ideal in an AG $k$-algebra. We illustrate this by giving a complete list of $2$-admissible sequences of length at most $3$ and with $h_0=2$ that are realizable by an ideal in an AG $k$-algebra.
\end{abstract}

\maketitle

\tableofcontents
\addtocontents{toc}{\protect\setcounter{tocdepth}{1}}
\section{Introduction}
Let $(R, \mathcal{M})$ be a local ring over a field $k$ where $k = R/\mathcal M$ and $J$ an ideal in $R$ such that $A=R/J$ is an Artinian Gorenstein (AG) $k$-algebra. One of the most important numerical invariants associated with an AG $k$-algebra is its Hilbert function. %We denote the Hilbert function of a graded $k$-algebra by $\HF.$
Recall that the Hilbert function of $A$ is by definition the Hilbert function of its {\it tangent cone}, namely $G(\m)=\oplus_{n \geq 0}\m^n/\m^{n+1}$ where $\m=\mathcal{M}/J.$
Classifying the Hilbert function of AG $k$-algebras is a widely open problem in commutative algebra. 
They have been characterized in certain cases, see \cite{Br77,I84Compressed,JMR23,Macaulay1904,MR18}.
The purpose of this article is to understand the structure of the Hilbert function of $I,$ namely, the Hilbert function of its {\it associated graded ring} $G(I):=\oplus_{n \geq 0}I^n/I^{n+1}$ for any ideal $I$ in $A.$ The associated graded ring $G(I)$ plays an important role in algebraic geometry, in particular in resolution of singularities. There has been a huge research in the literature on studying the properties of $G(I)$ in terms of the properties of $I$ and a local ring $A$ of any dimension (not just AG $k$-algebra), see \cite{CPR05, GN94, Hun87, PU99, RV96, RV10, Sally77, V94}.   
\vskip 2mm

For a standard graded $k$-algebra $S=\bigoplus_{i \geq 0}S_i$ the Hilbert function $\HF: \mathbb{Z} \to \mathbb{Z}_{\geq 0}$ defined as $\HF(S)_i := \ell_{S_0}(S_i).$
It is well known that the Hilbert function of a {\it graded} AG $k$-algebra is symmetric.
This is no longer true with the Hilbert function of the local AG $k$-algebra.
However, A. Iarrobino in \cite{I94Memoir} developed a beautiful theory that shows the Hilbert function of an AG $k$-algebra admits a symmetric decomposition, namely, 
$$\HF(G(\m))_i = \sum_{a = 0}^{j-2} H(a)_i,$$
    where $H(a)$ is symmetric about $(j-a)/2.$ Here $j$ is the {\it socle degree} of $A,$ that is the maximum integer such that $\m^j \neq 0.$ This became a very powerful tool in classifying the Hilbert functions of local AG $k$-algebras (see for example \cite{MR18}). Recently, there has been intense research on symmetric decomposition, for instance in \cite{IP21} the authors have studied the possible vectors $H(a)$ in a symmetric decomposition of the Hilbert function of AG $k$-algebra, in \cite{IP21_B} the authors have studied its interaction with Jordan type.
 Very recently, M. Wojtala has extended the symmetric decomposition theory to self-dual modules in \cite{W24}. 
 
In Section 3, we introduce a symmetric decomposition of the Hilbert function of an arbitrary ideal in an AG local $k$-algebra. We prove that:

\begin{customthm}{1} (Theorem \ref{Thm:SymmetricDecomposition})
\label{Thm:MainIntro}
 Let $(A, \m)$ be an AG $k$-algebra, and $I$ an ideal in $A$ such that $I^u \neq 0$ and $I^{u+1} = 0$. Then the Hilbert function of $I$ $$\HF(G(I))_i = \sum_{a = 0}^u H(a)_i,$$
    where $H(a)$ is symmetric about $(u-a)/2.$
 \end{customthm}
 As a consequence, we recover a result by W. Heinzer, M.-K. Kim, and B. Ulrich in \cite{HKU11} which asserts that $G(I)$ is Gorenstein if and only if the Hilbert function of $G(I)$ is symmetric (Theorem \ref{thm:HKUresult}).
\vskip 2mm

Our hope is that Theorem \ref{Thm:MainIntro} will be useful in classifying the Hilbert function of an ideal in an AG $k$-algebra. One of the difficulties in understanding the structure of the Hilbert function of $G(I)$ is that, in this case $G(I)_0=A/I,$ which is not a field unless $I=\m.$
In \cite{BN99} C. Blancafort and S. Nollet characterized the Hilbert functions of standard graded algebras of the form $R_0[x_1,\ldots,x_b]/I$ where $R_0$ is an Artinian ring. They call such a sequence a {\it $b$-admissible sequence}. This result generalizes a result of Macaulay which characterizes the Hilbert functions of standard graded algebras over a field \cite{Macaulay1927}. We ask: which $b$-admissible sequences occur as the Hilbert function of $G(I)$ for some ideal $I$ in an AG $k$-algebra $A?$ We call a sequence that occurs as the Hilbert function of $G(I)$ for some ideal $I$ in an AG  $k$-algebra $A$ as an {\it $I$-Gorenstein sequence}.  Recall that a Gorenstein sequence refers to a sequence that occurs as the Hilbert function of some AG $k$-algebra. To the best of our knowledge, this is the first time realizable Gorenstein sequences are studied in this generality.  
\vskip 2mm
As with the case of Gorenstein sequences, the classification of $b$-admissible sequences that are $I$-Gorenstein is a difficult problem. However, we want to investigate the problem in certain cases. For instance, the structure of the Hilbert function of the AG $k$-algebra of codimension two is well known \cite{Macaulay1904, Macaulay1916}. Inspired by this, in this paper we investigate the structure of Hilbert function of an ideal in a codimension two AG $k$-algebra. 
But we are far from having a complete solution. In Section 4, we propose some problems and conjectures on possible $I$-Gorenstein sequences in a codimension two AG $k$-algebra. 
In Section 5, we give a complete list of $2$-admissible sequences of length at most $3$ and with $h_0=2$ that are $I$-Gorenstein sequences which supports our conjecture. To conclude that certain $2$-admissible sequences are not $I$-Gorenstein, we have deeply extracted information coming from a possible symmetric decomposition of this sequence. In order to show that a certain $2$-admissible sequence $H$ is an $I$-Gorenstein sequence we have constructed an explicit AG $k$-algebra $A$ and an ideal $I$ in $A$ such that $G(I)$ has the Hilbert function $H,$ see Table \ref{table:IGorenstein}.

In \cite[Theorem 2]{I89} Iarrobino showed that the Hilbert function of codimension $2$ AG $k$-algebra determines its symmetric decomposition. In Section 4, we provide an example of a $2$-admissible $I$-Gorenstein sequence with two distinct admissible symmetric decompositions (Example \ref{Example:2SymmDec}).
 \vskip 2mm
We have recalled some properties of bilinear maps,  subspaces associated with them, and properties of AG $k$-algebras in Section 2. 
\vskip 2mm
We have used computer algebra systems CoCoA (see \cite{CoCoA}), Macaulay2 (see \cite{M2}), Singular (see \cite{Singular}), and the library (\cite{EliasCode}) for various computations in this paper.
\vskip 2mm
\noindent {\bf Acknowledgments}: The authors thank M.E. Rossi for helpful discussions and M. Pedro for sharing the Macaulay2 code for computing the symmetric decomposition. We have greatly benefited from the examples computed using this code in Macaulay2. 

\section{Preliminaries}
In this section, we recall some properties of bilinear maps and the subspaces associated with them that we need in the next section. These properties are well-known, but we have added details of some results for the sake of completeness. 
One of the important properties of AG $k$-algebras is that one can associate a $k$-bilinear map with it. We recall this correspondence in this section. We also recall the definition of a $b$-admissible sequence introduced by Blancafort and Nollet in \cite{BN99}.

\subsection{Properties of bilinear maps}
\noindent Let $k$ be a field and $V, W$ be finite-dimensional $k$-vector spaces. Recall that a map $\varphi: V \times W \to k$ is said to be bilinear if it satisfies the following properties: for any $x,y \in V, z,w \in W $ and $\alpha \in k$, 
\begin{enumerate}[(i)]
  \item $\varphi(\alpha(x+y), z) = \alpha\,\varphi(x, z) +\alpha \, \varphi(y, z)$;
  \item $\varphi(x, \alpha(z+w)) = \alpha \,\varphi(x, z) + \alpha \, \varphi(x, w)$. 
\end{enumerate}

Let $\varphi : V \times W \to k$ be a bilinear map. For a subset $S$ of $ V ,$ we define a subset $S^\perp$ of $W$ as
  \[
   S^{\perp}:= \{ w \in W \mid \varphi(v, w) = 0  ~\textrm{for every} ~v \in S \}.
  \]
  Similarly, if $T$ is a subset of $W,$ then a subset $\ann_V(T)$ of $V$ is
\[\ann_V(T) := \{ v \in V \mid \varphi(v, t) = 0  ~\textrm{for all} ~t \in T \}.\]  

Note that both $\ann_V(T)$ and $S^{\perp}$ are subspaces of $V$ and $W,$ respectively. For subsets $S,O$ of $V$ recall that \[S+O:=\{x+y:x \in S \mbox{ and } y \in O\}.\]

\begin{proposition}\label{Prop:PerpAnn}
    Let $\varphi : V \times W \to k$ be a bilinear map and $S, O \subseteq V ,$ and $T , U \subseteq W$ are subsets.
    Then 
    \begin{enumerate}[(i)]
        \item If $T\subseteq U$, then $\ann_V(U) \subseteq \ann_V(T)$.
        \item If $S \subseteq O$, then $O^{\perp} \subseteq S^{\perp}$.
        \item $\ann_V(T+U) = \ann_V(T)\cap \ann_V(U).$
        \item $(S+O)^{\perp} = S^{\perp} \cap O^{\perp}$.
        
    \end{enumerate}
\end{proposition} 
\begin{proof} Easy to check.
\end{proof}
%--------------------------------------
Throughout this paper, for any $k$-vector space $V$, we set $V^{\widecheck{}} :=\Hom(V,k).$ For a subset $S$ of $V,$ let 
\[
( 0\colon_{V^{\widecheck{}}} S):=\{f\in {V^{\widecheck{}}}:f(t)=0 \mbox{ for all } t \in S\}.
\] Note that $( 0\colon_{V^{\widecheck{}}} S)$ is a subspace of ${V^{\widecheck{}}}.$ 
We denote the  $k$-vector space  dimension of $V$ by $\dim V.$

\begin{proposition}\label{Prop:ColonPerp}
    Let $V,W$ be finite dimensional vector spaces of the same dimension, and $S\subseteq V, T\subseteq W$ are subspaces. Let $\varphi: V\times W \to k$ be a bilinear map. 
    \begin{enumerate}[(i)]
        \item If the $k$-linear map  $\varphi' : V \to \Hom(W,k)$ induced by $\varphi$ is injective, then
        \[( 0\colon_{W^{\widecheck{}}} T) \cong \ann_V(T).\]
        Moreover, $\dim \ann_V(T)=\dim W-\dim T.$
        \item If the $k$-linear map $\varphi'': W \to \Hom(V,k)$ induced by $\varphi$ is injective, then
        \[( 0\colon_{V^{\widecheck{}}} S) \cong S^{\perp}.\]
        Moreover, $\dim S^{\perp} = \dim V - \dim S.$
        \item If $\varphi$ is non-degenerate, then
       \[ \ann_{V}(S^{\perp}) = S,\, (\ann_{V}T)^{\perp} = T.\]
       \item Assume $\varphi$ is non-degenerate. Let $S, O$ be subspaces of $V$ and $T, U$ be subspaces of $W$. Then 
       \[(S\cap O)^{\perp} = S^{\perp} + O^{\perp}, \, \ann_V(T \cap U ) = \ann_V(T) + \ann_V(U).\]
    \end{enumerate}
\end{proposition}

\begin{proof}
    (i) Since $\dim V = \dim W$, $\dim V=\dim W^{\widecheck{}}.$ Therefore, if \[ \varphi' : V \to \Hom(W,k)=W^{\widecheck{}}\] induced by $\varphi$ is injective, then $\varphi'$ is surjective as well. Hence $\varphi'$ is an isomorphism.
     
    \noindent Let $\psi$ be the restriction map of $\phi'$ to $\ann_V(T)$, that is $ \psi:\ann_V(T) \to W^{\widecheck{}}$ is defined as $\psi(v)=\phi'(v).$ Note that for $v \in \ann_V(T)$, 
    \[
    \phi'(v)(t)=\phi(v,t)=0 \mbox{ for all } t \in T. 
    \] 
    Hence $\psi(v)=\phi'(v)\in ( 0\colon_{W^{\widecheck{}}} T).$ Therefore $\im(\psi) \subseteq ( 0\colon_{W^{\widecheck{}}} T)$. 
    % Hence $\psi :\ann_V(T) \to  ( 0\colon_{W^{\widecheck{}}} T)$ is well-defined. 
    Clearly, as $\phi'$ is injective, $\psi$ is injective. Also, for any $f \in ( 0\colon_{W^{\widecheck{}}} T),$ surjectivity of $\phi'$ implies that there exists $v \in V$ such that $\phi'(v)=f.$ Then $v \in \ann_V(T)$ because
    \[
    \phi(v,t)=\phi'(v)(t)=f(t)=0 \mbox{ for all } t \in T.
    \]
    Thus $\psi(v)=\phi'(v)=f.$ Hence $\im(\psi) =( 0\colon_{W^{\widecheck{}}} T) $. Thus we get $\ann_V(T) \cong ( 0\colon_{W^{\widecheck{}}} T).$

\noindent   In order to compute the dimension of $\ann_V(T),$ first observe that $(0\colon_{W^{\widecheck{}}} T) \simeq (W/T)^{\widecheck{}}$. Hence 
\[
\dim \ann_V(T)=\dim  (0\colon_{W^{\widecheck{}}} T)=\dim (W/T)^{\widecheck{}}=\dim W/T=\dim W-\dim T.
\]
    \vskip 2mm
    
\noindent (ii) Proof is similar to (i). 
 \vskip 2mm
 \noindent (iii) 
We have
\begin{eqnarray}\label{EQ:annPerp}
    \dim \ann_V(S^{\perp}) &=& \dim W - \dim S^{\perp} \hspace{2.2cm}\mbox{ (from (i))} \nonumber\\
    &=& \dim W-(\dim V - \dim S) \hspace{0.9cm} \mbox{(from (ii))} \nonumber\\
    &=&\dim S \hspace{4cm}    \mbox{(since $\dim W=\dim V$)}. \nonumber
\end{eqnarray}
Since $ S \subseteq \ann_V(S^{\perp}) $, we obtain that $\ann_V(S^{\perp}) =S.$

\noindent Similarly, $(\ann_{V}T)^{\perp} = T$ is proved.

\noindent (iv) We have
\begin{equation*}
    \begin{split}
        \ann_V(S^{\perp} + O^{\perp})& = \ann_V(S^{\perp}) \cap  \ann_V(O^{\perp}) \quad (\text{by Proposition } \ref{Prop:PerpAnn} \,(iii)) \\
                                    & =  S \cap O \quad \hspace{2.9cm} (\text{by part (iii)}).
    \end{split}
\end{equation*}
Therefore we obtain
\begin{equation*}
    \begin{split}
         (S \cap O)^{\perp} &=(\ann_V(S^{\perp} + O^{\perp}))^{\perp} \\
        &=S^{\perp} + O^{\perp}                    \quad (\text{by part (iii)}).
    \end{split}
\end{equation*}
Similarly, $ \ann_V(T \cap U ) = \ann_V(T) + \ann_V(U)$ is proved.

\end{proof}

%%%%%%%%%%%%%%%%%%%%%%%%%%%%%%%%%%%%%%%%%%%%%%%%%%%%%%%%%%%%%%%%%%%%%%%%%%%%%%%%%%%%%%%%%%%%%%%%%%%%%%%%%%%%%%
\subsection{AG $k$-algebras}
%%%%%%%%%%%%%%%%%%%%%%%%%%%%%%%%%%%%%%%%%%%%%%%%%%%%%%%%%%%%%%%%%%%%%%%%%%%%%%%%%%%%%%%%%%%%%%
Let $(R, \mathcal{M})$ be a local ring over a field $k \cong R/\mathcal{M}.$ In this subsection we recall that there is a one-many correspondence between AG quotients of $R$ and $k$-linear maps from $R \to k$ satisfying certain conditions. This map induces a non-degenerate bilinear map from $A \times A \to k$ which is a key ingredient to derive our results in the next section. 

\begin{proposition}\label{Prop:Lmapcorr} 
Let $(R, \mathcal{M})$ be a local ring over a field $k$. 
There is a one-many correspondence between the following sets:\\
\[
\left\{
\begin{array}{c}
\text{AG quotients } A \text { of } R \\ \text{of socle degree s
}
\end{array}
\right\}
\xleftrightarrow{}
\left\{
\begin{array}{c}
k \text{-linear map } \varphi: R \to k \text{ such that } \\ \varphi\mid_{\mathcal{M}^{s+1}} =0, \varphi\mid_{\mathcal{M}^{s}} \neq 0. 

\end{array}
\right\}
\]
Here $A=R/J$ with $J=\{h \in R \mid \varphi(Rh)=0\}.$
\end{proposition}
    
\begin{proof}
Refer to \cite[Lemma 1.1]{I94Memoir}.        
\end{proof}
%%%%%%%%%%%%%%%%%%

%%%%%%%%%%%%%%%%%%
%----------------------------------------------    
\begin{lemma}\label{Lem:nonDegenerate}
    Let $A=R/J$ be an AG $k$-algebra of socle degree $s$. Then there exists a $k$-linear map $\varphi: A \to k$ which is non-zero on the socle of $A$. The $k$-bilinear map induced by this linear map $\langle \cdot, \cdot \rangle: A \times A \to k$, namely $\langle a, b \rangle = \varphi(ab)$, is non-degenerate.
\end{lemma}
  
\begin{proof} By Proposition \ref{Prop:Lmapcorr} corresponding to $A=R/J$, there exists a $k$-linear map $\psi \colon R \to k$ such that $\psi\mid_{\mathcal{M}^{s+1}} =0, \psi\mid_{\mathcal{M}^{s}} \neq 0 $. Define $\varphi: A \to k$ as $\varphi(a+J)=\psi(a)$.
By Proposition \ref{Prop:Lmapcorr}, we have $\psi(Rx)=0$ for all $x\in J.$ Hence, in particular $\psi(x)=0$ for all $x \in J.$ Thus $\varphi$ is well defined. Moreover, since $\text{Soc } A = \mathcal{M}^s+J/J $, $\varphi(\mathcal{M}^s+J/J ) = \psi(\mathcal{M}^s) \neq 0$. Hence $\varphi$ is non-zero on the socle of $A.$
        Clearly, the induced map $\langle \cdot , \cdot \rangle : A\times
   A \to k$ given by $\langle a, b \rangle = \varphi(ab)$ is $k$-bilinear.
        
         We prove that $\langle \cdot , \cdot \rangle$ is non-degenerate. 
        Suppose $a + J \in A= R/J$ is such that $\langle a+J, b+J \rangle = 0,$ for all $b \in R.$  Then, for all $b \in R,$
        \begin{equation*}
            \langle a+J, b+J \rangle =  \varphi(ab+J) = \psi(ab)=0.
        \end{equation*}
        Hence by the description of $J$ in Proposition \ref{Prop:Lmapcorr}, we get that $a \in J$. Hence $a+J = 0.$ 
        \newline Similarly, if $b+J \in A$ is such that $ \langle a+J, b+J \rangle = 0$ for all $a+J \in A,$ then  $b+J = 0$ in $A$.
        Therefore our claim is proved. 
  \end{proof}    

    By abusing the notation we denote the bilinear map  $\langle \cdot , \cdot \rangle$ in Proposition \ref{Lem:nonDegenerate} which is induced by the linear map $\varphi:A \to k$ as $\varphi$. We recall the definition of perp and annihilator of an ideal $I$ in $A$ induced by this bilinear map $\varphi:$ 
    \[
    \ann_A(I)=\{a \in A\mid \varphi(at)=0, \text{ for all } t \in I\},\, \text{ and } \,
    I^{\perp} = \{a \in A\mid \varphi(ta)=0, \text{ for all } t \in I\}
     \]
    Hence, we get $\ann_A(I) = I^{\perp}$.

\begin{proposition}\label{Prop:PerpAndCheck}
  Let $A=R/J$ be an AG $k$-algebra, and $Q \subseteq S \subseteq A$ be ideals in $A$. Then ${(S/Q)}^{\widecheck{}} \cong Q^{\perp}/S^{\perp}.$
\end{proposition}
\begin{proof} Let $\varphi:A \to k$ be a $k$-linear map corresponding to $A$ obtained from Lemma \ref{Lem:nonDegenerate}.
    We have a natural $k$-linear map $\psi: Q^{\perp} \to {(S/Q)}^{\widecheck{}}$
    given by $$\psi(a) = \psi_a : S/Q \to k$$ where $\psi_a(s+Q) = \varphi(sa)$ for $s \in S.$ Since $a \in Q^\perp,$ $\psi_a$ is well defined.
     Note that 
     $$\ker \psi = \{a \in Q^{\perp} \mid \psi(a)=\psi_a=0\} = \{a \in Q^{\perp} \mid \varphi(sa)=0 \text{ for all } s \in S\} = S^{\perp} \cap Q^{\perp} = S^{\perp}.$$

    \noindent Hence, we have an injective map $\overline{\psi}:$ $Q^{\perp}/S^{\perp} \to {(S/Q)}^{\widecheck{}}$. Further,
    \begin{equation*}
        \begin{split}
            \dim(Q^{\perp}/S^{\perp}) & = \dim Q^{\perp} -\dim{S}^{\perp} \\
                                      & = \dim A - \dim Q - (\dim A - \dim S) \quad (\text{by Proposition } \ref{Prop:ColonPerp}) \\
                                      & = \dim S - \dim Q \\
                                      & = \dim (S/Q) \\
                                      & = \dim (S/Q)^{\widecheck{}}.
        \end{split}
    \end{equation*}
    % $$\dim(Q^{\perp}/S^{\perp}) = \dim Q^{\perp} -\dim{S}^{\perp} = \dim S-\dim Q = \dim {(S/Q)}^{\widecheck{}}.$$
    Hence $\overline{\psi}$ is an isomorphism. Thus ${(S/Q)}^{\widecheck{}} \cong Q^{\perp}/S^{\perp}$.
\end{proof}

\subsection{$b$-admissible sequences} In this section, we recall a result by C. Blancafort and S. Nollet which characterizes the Hilbert functions of standard graded algebras over an Artinian local ring that need not be a field.
We remark that possible Hilbert functions of standard graded algebras over a field are classically known due to Macaulay \cite{Macaulay1927}. 

\begin{theorem}[\cite{BN99}, Theorem 2.3]\label{Thm:BN96}
Let $R_0$ be an Artinian ring, $R=R_0[x_1, \ldots, x_b]$, and $H: \mathbb{N} \to \mathbb{N}$ be a function. Define the functions $q$ and $r$ by the Euclidean division as $H(n)=
\binom{n+b-1}{b-1} q(n) + r(n)$ where $0\leq r(n) < \binom{n+b-1}{b-1}.$ Then $H$ is the Hilbert function of $R/I$ for a homogeneous ideal $I \subset R$ if and only if
\begin{enumerate}[(i)]
    \item $H(0) \leq \ell(R_0)$;
    \item $H(n+1) \leq \binom{n+b}{b-1}q(n) + {r(n)_n}_{+}^{+}$ for all $n \geq 0.$ 
\end{enumerate}
Here for the $n^{th}$ Macaulay representation ${r(n)_n} = \binom{k(n)}{n} + \cdots + \binom{k(1)}{1}$ of $r(n),$  
$${r(n)_n}_{+}^{+}:=\binom{k(n)+1}{n+1} + \cdots + \binom{k(1)+1}{1+1}.$$ 
%if is the (See \cite{BH98})
\end{theorem}
%%%%%%%%%%%%%%%%%%%%%%%%%
\noindent Note that the integer $b$ is the embedding dimension of the ring $R$. A numerical function $H$ satisfying the equivalent condition above is called a {\it $b$-admissible sequence}. 

%%%%%%%%%%%%%%%%%%%%
\section{Symmetric decomposition of the Hilbert function of $G(I)$}
Let $(R, \mathcal{M})$ be a local ring over a field $k$ with  $k = R/\mathcal{M}$ and $J$ be an ideal in $R$ such that  $A=R/J$ is an AG $k$-algebra. Let $I$  be an ideal in $A$. We define the {\it end degree of $I$} as the least integer $u$ such that $I^u \neq 0 $ and $I^{u+1}=0$. 
Then we define the following decreasing sequence of ideals of $G(I)$ as follows:
$$G(I)=C(0) \supseteq C(1) \supseteq \cdots \supseteq C(u+1)=  0$$
where $C(a)$ is a graded ideal of $G(I)$ with $i^{th}$ component given by 
\begin{equation}\label{EQ:C(a)i}
    C(a)_i = \dfrac{(0:I^{u+1-a-i}) \cap I^i}{(0:I^{u+1-a-i}) \cap I^{i+1}} \cong \dfrac{(0:I^{u+1-a-i}) \cap I^i +I^{i+1}}{I^{i+1}}.
\end{equation}

\noindent Define $Q(a) := C(a)/C(a+1)$. 
We note that
\begin{equation*}
    \begin{split}
        Q(a)_v  & = C(a)_v /C(a+1)_v \\
                & = \dfrac{[(0:I^{u+1-a-v}) \cap I^v +I^{v+1}]/I^{v+1}}{[(0:I^{u+1-a-1-v}) \cap I^v +I^{v+1}]/I^{v+1}}.
    \end{split}
\end{equation*}
For simplification, let $l = u+1-a-v.$ Then
\begin{equation}\label{EQ:QDefinition}
    \begin{split}
        Q(a)_v & = \dfrac{(0:I^{l}) \cap I^v +I^{v+1}}{(0:I^{l-1}) \cap I^v +I^{v+1}}.
    \end{split}
\end{equation}
Since $(0:I^{l-1}) \subseteq (0:I^l),$  we can write 
\[(0:I^{l}) \cap I^v +I^{v+1} =(0:I^{l}) \cap I^v  + (0:I^{l-1}) \cap I^v +I^{v+1}.\]
Therefore
\begin{equation}\label{EQ:Q(a)}
    \begin{split}
        Q(a)_v & = \dfrac{(0:I^{l}) \cap I^v  + (0:I^{l-1}) \cap I^v +I^{v+1}}{(0:I^{l-1}) \cap I^v +I^{v+1}}.
    \end{split}
\end{equation}
Let $M := (0:I^{l}) \cap I^v$, $N := (0:I^{l-1}) \cap I^v +I^{v+1}$. Then $Q(a)_v = (M+N)/N \cong M/(M\cap N).$

\noindent Now 
\begin{equation}\label{EQ:MandN}
    \begin{split}
        M \cap N & = \left( (0:I^l) \cap I^v \right) \cap \left( I^{v+1} + (0:I^{l-1}) \cap I^v \right) \\
                & = \left[ (0:I^l) \cap I^v \cap I^{v+1} \right] +  \left[ (0:I^{l-1}) \cap I^v \right] \quad (\mbox{since } (0:I^l) \cap I^v \supseteq (0:I^{l-1}) \cap I^v) \\
                & = (0:I^l) \cap I^{v+1} + (0:I^{l-1}) \cap I^v.
    \end{split} 
\end{equation}
Hence
\begin{equation}\label{EQ:Q}
    \begin{split}
        Q(a)_v & \cong \dfrac{I^v \cap \, (0: I^l)}{I^{v+1} \cap \, (0:I^l) +I^v \cap \, (0:I^{l-1})}
    \end{split}
\end{equation}
where $l = u+1 -a-v.$

\noindent The following lemma is a key tool for proving our main results. 
\begin{lemma}\label{Lemma:colonPerpLemma}
  Let $A=R/J$ be an AG $k$-algebra,  $I$ an ideal of $A$, and $\phi$ be a map corresponding to $A$ as in Lemma \ref{Lem:nonDegenerate}. Then $\ann_A(I)=I^{\perp} = (0:I)$, and $I = (0:I)^{\perp}.$
\end{lemma}
\begin{proof}
  By definition $\ann_A(I) = \{ a \in A \mid \varphi(aI) =0 \}.$ 
  Let $a \in \ann_A(I).$ Then
  \[ \varphi(aIA) = \varphi(aI) =0.\] 
  Since $\varphi $ is non-degenerate, we get $aI = 0$. Hence $a \in (0 : I)$. 
  \par Conversely, if $x  \in (0 : I)$, then $\varphi(xI) = \varphi(0) = 0$. Therefore $x \in \ann_A(I)$. Hence $\ann_A(I)= (0:I)$. As $\ann_A(I)=I^{\perp}, $ we get $I^\perp=(0:I).$ 
  
  Thus $(\ann_A(I))^\perp=(0:I)^\perp.$ Since $(\ann_A(I))^\perp=I$ by Proposition \ref{Prop:ColonPerp} (ii), $I=(0:I)^\perp.$  
\end{proof}
%%%%%%%%%%%%
\noindent Next, we prove the main technical result that shows that the Hilbert function of $Q(a)$ is symmetric about $(u-a)/2.$

\begin{proposition}\label{Prop:technical}
  Let $A=R/J$ be an AG  $k$-algebra, $I$  an ideal in $A$ with the end degree $u$. Then for any $a \geq 0$, 
   \begin{equation*}
        {Q(a)_v}^{\widecheck{}}:= \Hom(Q(a), k)_v \cong Q(a)_{u-a-v}.
      \end{equation*}
\end{proposition}  
\begin{proof}
  Let $l = u+1 - (a+v)$. By Equation \eqref{EQ:Q}, we have
  
  \begin{equation*}
    \begin{split}
      {Q(a)_v}^{\widecheck{}} & \cong {\left(\frac{I^v \cap \, (0: I^l)}{I^{v+1} \cap \, (0:I^l) +I^v \cap \, (0:I^{l-1})}\right)}^{\widecheck{}} \\
      & \cong \, \frac{{\left(I^{v+1} \cap \, (0:I^l) + I^v \cap \, (0:I^{l-1})\right)}^{\perp}}{{\left(I^v \cap \, (0:I^l)\right)}^{\perp}} \quad   \hspace{7em} (\text{by Proposition } \ref{Prop:PerpAndCheck}) \\
      & = \, \frac{{\left(I^{v+1} \cap \, (0:I^l)\right)}^{\perp} \cap \, {\left( I^v \cap \, (0:I^{l-1})\right)}^{\perp}}{{\left(I^v \cap \,  (0:I^l)\right)}^{\perp}} \quad \hspace{5.3em} (\text{by Proposition } \ref{Prop:PerpAnn} \text{ (iv)})\\
      & = \, \frac{\left((I^{v+1})^{\perp} + (0:I^l)^{\perp}\right) \cap \,  \left((I^v)^{\perp} + \, (0:I^{l-1})^{\perp}\right)}{(I^v)^{\perp} +  (0:I^l)^{\perp}} \quad \hspace{2.7em} (\text{by Proposition } \ref{Prop:ColonPerp}\text{ (iv)}) \\
      & = \, \frac{\left( (0:I^{v+1}) +I^l  \right) \cap \, \left( (0: I^v) + I^{l-1}\right)}{(0:I^v) +I^l} \quad \hspace{7.1em} (\text{by Lemma } \ref{Lemma:colonPerpLemma}) \\
      & = \,   \frac{(0:I^{v+1}) \cap \left[ (0: I^v) + I^{l-1}\right] + I^l}{(0:I^v) +I^l} \quad \hspace{8.3em}(\text{since } I^{l} \subseteq(0:I^v)+ I^{l-1}) \\
      & =\,  \frac{ (0: I^v) + I^{l-1} \cap \, (0: I^{v+1}) + I^l }{(0:I^v) +I^l}  \quad \hspace{9.1em} (\text{since } (0:I^v) \subseteq (0: I^{v+1}))\\
      &  = \,  \frac{  I^{l-1} \cap \, (0: I^{v+1}) +(0: I^v) + I^l}{(0:I^v) +I^l}   \\
      & \cong \, \frac{ I^{l-1} \cap \, (0: I^{v+1})}{\left[ (0: I^v) + I^{l}\right] \cap \, \left[I^{l-1}\cap \,(0:I^{v+1}) \right] } \\
      & = \, \frac{ I^{l-1} \cap \, (0: I^{v+1})}{\left[\left[ (0: I^v) + I^{l}\right] \cap \, (0:I^{v+1})\right] \cap \, I^{l-1} } \\
     &  = \, \frac{ I^{l-1} \cap \, (0: I^{v+1})}{ \,\left[  (0: I^v) + I^{l} \cap (0:I^{v+1})  \, \right]  \cap \, I^{l-1} } \quad \hspace{7.5em}(\text{since } (0:I^v) \subseteq (0: I^{v+1})) \\ 
    & = \, \frac{ I^{l-1} \cap \, (0: I^{v+1})}{ (0: I^v) \cap \, I^{l-1}  + I^{l} \cap (0:I^{v+1}) } \quad \hspace{8.5em}(\text{since } I^{l} \cap (0:I^{v+1}) \subseteq I^{l-1}) \\
           % & = \, \frac{ I^{l-1} \cap \, (0: I^{v+1})}{   I^{l-1} \cap \,(0: I^v)   + I^{l} \cap (0:I^{v+1}) } \\
           & =Q(a)_{l-1}. \quad \hspace{19em}(\text{by Equation }\ref{EQ:Q}) \\
           & =  Q(a)_{u-a-v} \quad \hspace{18em}(\text{since } l = u+1 - (a+v)).     
      \end{split}
   \end{equation*}

\end{proof}

\noindent The following is one of the main results of the paper:
\begin{theorem}\label{Thm:SymmetricDecomposition}
    Let $A =R/J$ be an AG $k$-algebra, and $I$ an ideal in $A$ with the end degree $u$. Then the Hilbert function of $G(I)$, $$\HF({G(I)})_i = \sum_{a= 0}^{u} H(a)_i,$$
    where $H(a)$ denotes the Hilbert function of $Q(a)$. Moreover, $H(a)$ is symmetric about $(u-a)/2.$
\end{theorem}
\begin{proof}
We have $G(I)=C(0) \supseteq C(1) \supseteq \cdots \supseteq C(u+1)=  0$. By considering the Hilbert function, we obtain $\HF({G(I)})_i = \sum_{a = 0}^{u} H(a)_i$.   Further, by Proposition \ref{Prop:technical}, we have \[Q(a)_v \cong Q(a)_{u-a-v}.\] 
Hence $H(a)_v = H(a)_{u-a-v}.$ This proves the result.
\end{proof}

\begin{definition}
   Let $H$ be a $b$-admissible sequence. We say $H$ is an {\it $I$-Gorenstein sequence} if there exists a Gorenstein ring $A$ and an ideal $I\subset A$ such that the Hilbert function of $G(I)$ is $H$.
\end{definition}

\begin{corollary}\label{Cor:HuAtmostH0}
    Let $H = (h_0, h_1, \ldots, h_u)$ be a $b$-admissible sequence. If $h_u > h_0$, then $H$ is not an $I$-Gorenstein sequence.
\end{corollary}
\begin{proof}
    Suppose $H$ is an $I$-Gorenstein sequence. Notice that by Equation \eqref{EQ:Q}, $Q(a)_u=0$ for all $a>0.$ Hence by Theorem \ref{Thm:SymmetricDecomposition}, $h_u = H(0)_u.$  By symmetry of $H(0)$, we get $h_u= H(0)_u = H(0)_0 $. Since $H(0)_0 \leq h_0$, we get $h_u \leq h_0$. This is a contradiction. 
\end{proof}

Next, we give an example to illustrate Theorem \ref{Thm:SymmetricDecomposition}.
\begin{example}\label{Example:1}
    Let $A = \frac{k[\![ t^4, t^5, t^6]\!] }{(t^{10})}$. We know that $R= k[\![ t^4, t^5, t^6]\!]$ is Gorenstein, since the numerical semigroup ring generated by 4,5,6 is symmetric (see \cite[Theorem 4.4.8]{BH98}). Further, $t^{10}$ is a nonzerodivisor in $R$. Hence $A$ is an AG $k$-algebra. Let $I = (t^6)$ be an ideal in $  A$. 
    Then $I^2 = (t^{12})$, and $I^3 =0$. Hence the end degree of $I$ is 2.  
     \par Note that $\ell(A) = 10, \, \ell(A/I) = 6, \, \ell(A/I^2)  = 8$. Hence the Hilbert function of $G(I)$ is (6,2,2).
    Now we compute $C(a)$ to obtain $Q(a)$.
    We have 
\begin{equation*}
  \begin{split}
    (0:I) & = (t^4, t^8, t^9, t^{12}, t^{13}, t^{17}), \\
    (0:I^2) & = (t^4, t^6, t^8, t^9,  t^{11}, t^{12}, t^{13}, t^{17}).
  \end{split}
\end{equation*}
\begin{center}
\begin{tabular}{|c|c|c|}
  \hline
      $\cap$     & $I$ & $I^2$ \\
        \hline
  $(0:I)$ &   $( t^{12}, t^{17})$ & $I^2$ \\
  \hline
  $(0:I^2)$ &  $I$ & $I^2$\\
  \hline
\end{tabular}
\end{center}
Therefore 
\begin{equation*}
  \begin{split}
    C(1)_0 & = \frac{(0:I^2) \cap A}{(0:I^2) \cap I} = \frac{(0:I^2)}{I} \\
    C(1)_1 & = \frac{(0:I) \cap I}{(0:I) \cap I^2} = \frac{( t^{12}, t^{17})}{I^2}=0 \\
    C(1)_2 & = 0 \\
    C(2)_0 & = \frac{(0:I) \cap A}{(0:I)\cap I} =\frac{(0:I) +I}{ I} = \dfrac{(t^4,t^6,t^{8},t^9, t^{11}, t^{12}, t^{13}, t^{17})}{I} =  \frac{(0:I^2)}{I}\\
    C(2)_1 & = C(3)_0 = 0.
  \end{split}
\end{equation*}
Hence \begin{equation*}
  \begin{split}
    Q(0)_0 & = C(0)_0/C(1)_0 \cong A/(0:I^2) \\
    Q(0)_1 & = C(0)_1/C(1)_1 \cong I/I^2 \\
    Q(0)_2 & = C(0)_2/C(1)_2 \cong I^2 \\
    Q(1)_0 & =C(1)_0/C(2)_0 = 0 \\
    Q(1)_1 & = Q(1)_2 = 0 \\
    Q(2)_0 & = C(2)_0/C(3)_0 \cong \dfrac{(0:I^2)}{I} \\
    Q(2)_1 & =Q(2)_2= 0.
  \end{split}
\end{equation*}

Therefore we have a decomposition of the Hilbert function given below:
\begin{center}
  \begin{tabular}{ c|c c c}
    
        \hline
     {\bf HF}  & {\bf 6 } & {\bf 2} & {\bf 2} \\
   \hline
    H(0) & 2 & 2& 2 \\

    H(1) & 0 & 0 & 0 \\
  
    H(2) & 4 & 0 & 0 \\
    \hline
  \end{tabular}
\end{center}
The Hilbert series computation of $G(I)$ can also be verified in COCOA as follows:
\begin{verbatim}
    Use S::=QQ[x,y,z];
    J:=Ideal(y^2-x*z,y^2,x*z,x^3-z^2);
    I:=Ideal(z)+J;
    PS:= PrimaryHilbertSeries(J,I); PS;
\end{verbatim} \qed
\end{example}  
Next, we show that, as in the case of symmetric decomposition of $G(\m)$ (see [Theorem 1.5]\cite{I94Memoir}), $Q(0)$ is a Gorenstein quotient of $G(I)$ of socle degree $u$. 

\begin{proposition}\label{Prop:Q(0)Gorenstein}
    Let $A=R/J$ be an AG $k$-algebra, and $I$ an ideal in $A$ with the end degree $u$. Then $G(I)/C(1) \cong Q(0)$ is a graded Gorenstein quotient of $G(I)$ of socle degree $u$.  
\end{proposition}
\begin{proof}
    First we prove that $Q(0)$ is a Gorenstein $k$-algebra.
    Consider $\mathfrak{M} = \frac{\m}{(0:I^u)} +Q(0)_{i\geq 1}$, the maximal ideal of $Q(0)$. We prove $(0: \mathfrak{M})$ is one-dimensional. Let $i <u$ and $ \bar{a} \in Q(0)_i$ be such that $\bar{a} \mathfrak{M}=0. $ Then $\bar{a}\cdot Q(0)_1 = 0$. Hence by Equation \eqref{EQ:QDefinition}, we obtain 
    $$a \cdot I \subseteq (0:I^{u-{i-1}}) \cap I^{i+1} + I^{i+2} \subseteq  (0:I^{u-{i-1}}). $$
    Hence $$a \in (0:I^{u-{i-1}}) :I = (0:I^{u-{i}}).$$
    This gives $\bar{a} = 0$. Therefore $(0: \mathfrak{M})$ has no element of degree $< u.$ 
  
    \noindent Now suppose $ \bar{x} \in Q(0)_u$ is such that $\bar{x} \in (0:\mathfrak{M}).$  Then $\bar{x} \in \left(0: \frac{\m}{(0:I^u)}\right)$. Again, by using Equation \eqref{EQ:QDefinition}, we get
    $$x \m \subseteq (0: I^{0}) \cap I^u = 0.$$  Therefore $x \in (0:\m)$. Hence we obtain $(0: \mathfrak{M})_u = (0:\m) \cap I^u$. Since $A$ is a Gorenstein local ring, and an essential extension of $(0:\m)$, we obtain that $(0: \mathfrak{M})_u = (0:\m) \cap I^u$ is a one-dimensional $k$-vector space. Hence $Q(0)$ is Gorenstein of socle degree $u$.
\end{proof}  
% An important point here is that unlike the case of the symmetric decomposition of $G(\m)$, the following example shows that $Q(0)$ is not the only graded Gorenstein quotient of $G(I)$ of socle degree $u.$
% \begin{example}
% Consider the ring $A$ and the ideal $I$ as in the Example \ref{Example:1}.  Let $L =(t^{17})$ an ideal in $G(I).$  
% Note that as a graded ideal, $L_0 = 0, L_1 = 0, L_2 = (t^{17}).$  Moreover, $Q(0) \not\cong G(I)/L$. Indeed, $C(1)_2 = 0$ and $L_2 \neq 0$, hence $Q(0)_2 \not \cong \left( G(I)/L\right)_2.$ 

% We prove that $G(I)/L$ is graded Gorenstein of socle degree $u=2.$ For that we prove the socle in one-dimensional. 
% The maximal ideal in $G(I)/L$ is $\mathcal{M} = \frac{\m/I}{L_0} + \frac{I/I^2}{L_1} + \frac{I^2}{L_2}.$ 
% Let $\bar{x} \in \left(G(I)/L\right)_i$ for $i = 0,1$ such that $\bar{x} \in (0 \colon \mathcal{M}).$ 
% In particular, since $L_0=L_1 = 0,$
% \[\bar{x} \cdot \m/I \subseteq (0).\]
% Hence $x \m \subseteq I.$ Therefore $x \in (I : \m).$ By computation, we have $(I \colon \m) = I$. Hence $\bar{x} = 0.$ 

% Let $\bar{x} \in \left(G(I)/L\right)_2$ such that $\bar{x} \in (0 \colon \mathcal{M}).$ i.e., $x \in I^2.$ Then $\bar{x} \in (L_2 \colon \m).$ Hence \[(0 \colon \mathcal{M})_2 \subseteq \left(I^2 \cap (L_2 \colon \m)\right)/L_2.\] Conversely, $\left(I^2 \cap (L_2 \colon \m)\right)/L_2 \subseteq  (0 \colon \mathcal{M})_2$ is clear. Further, note that $I^2 = (t^{12}) \subseteq (L_2 \colon \m).$
% Therefore $\dim (0 \colon \mathcal{M}) = \dim \frac{I^2}{L_2} = \ell(A/L_2) -\ell(A/I^2) = 1.$ 
% \end{example}

Next, we provide equivalent criteria for $G(I)$ to be Gorenstein. Note that the equivalence of $(i)$ and $(v)$ in the following result was obtained in \cite[Theorem 4.2]{HKU11}.  
%In the following result $(i) \iff (v)$ is proved  
Here we provide an alternative proof of the same with some more equivalent conditions in terms of the vanishing of $Q(a)$ for $a>0.$ 
%using symmetric decomposition of Hilbert functions.
\begin{theorem}
\label{thm:HKUresult}
    Let $A =R/I$ be an AG  $k$-algebra, and $I$ an ideal in $A$ with the end degree $u$. Then the following are equivalent:
    \begin{enumerate}[(i)]
        \item $G(I)$ is Gorenstein of socle degree $u$;
        \item $C(1) =0$;
        \item $C(i)=0$ for all $i\geq 1$;
        \item $Q(i)=0$ for all $i \geq 1$;
        \item $\HF({G(I)})$ is symmetric.
    \end{enumerate}
\end{theorem}
\begin{proof} $(i) \iff (ii):$ Suppose $G(I)$ is Gorenstein of socle degree $u$. 
By definition \ref{EQ:C(a)i} we have
\[C(1)_i = \left( (0\colon I^{u-i}) \cap I^i +I^{i+1}\right) / I^{i+1}.\]
We claim that $(0\colon I^{u-i}) = I^{i+1}$  for all $0 \leq i \leq u.$ Hence $C(1)_i = 0$. Indeed,  
$I^{i+1} \subseteq (0\colon I^{u-i})$. Conversely, suppose there exists $\alpha \in (0\colon I^{u-i}) \setminus  I^{i+1}.$ Choose $j < i+1$ such that $\alpha \in I^{j}\setminus I^{j+1}.$ Consider the ideal $ J=( \alpha^*)$ in $G(I)$ where $\alpha^*$ denotes  the initial form of $\alpha$ in $G(I).$ Since $G(I)$ is an essential extension of $\Soc G(I)$ and  $G(I)$ is Gorenstein of socle degree $u,$, we get  $\left[\Soc G(I)\right]_u \cap J_u \neq 0.$ Therefore  $0\neq  \alpha^* x^* \in \Soc G(I) \cap J$ for some $x^* \in G(I)_{u-j}$. 
Since $x \in I^{u-j} \subseteq I^{u-i}$ and $\alpha \in (0\colon I^{u-i})$, we get $\alpha x = 0$ which is a contradiction. Therefore we have proved the claim.  

Conversely, $C(1)=0$ implies that $G(I)= Q(0).$ Hence $G(I)$ is Gorenstein of socle degree $u$ by Proposition \ref{Prop:Q(0)Gorenstein}.
\vskip 2mm

\noindent $(ii)\iff (iii):$ By definition $C(1) \supseteq C(i)$, for all $i \geq 2$. Hence if $C(1) = 0$, then $C(i)=0$ for all $i \geq 2.$ The converse is clear.
\vskip 2mm
\noindent $(iii) \iff (iv):$ By definition, $Q(i) = C(i)/C(i+1)$. Hence $Q(i)=0$ for all $i \geq 1$ if $C(i)=0$ for all $i \geq 1$. 
Conversely, suppose $Q(i)=0$, for all $i \geq 1$. Then
\[
C(1) =C(2) = \cdots = C(u) = C(u+1) = 0.
\]

\noindent $(iv) \iff (v):$ Suppose $Q(i)=0$ for all $i \geq 1.$ Since $\HF({G(I)}) = \sum_{i = 0}^u \HF({Q(i)})$ by Theorem \ref{Thm:SymmetricDecomposition}, $\HF({G(I)}) = \HF({Q(0)})$. By Theorem \ref{Thm:SymmetricDecomposition}, $\HF({Q(0)})$ is symmetric, and so does the Hilbert function of $G(I).$

 Conversely, suppose that the Hilbert function of $G(I)$ is symmetric. Note that the Hilbert function of $Q(0)$ is also symmetric. Further, $\HF({G(I)}) = \sum_{i= 0}^u \HF(a)$ by Theorem \ref{Thm:SymmetricDecomposition}. We claim that $Q(a)=0$ for all $a>0.$ Suppose $Q(a) \neq 0$ for some $a>0.$  
  Let 
 \[
 i:=\max\{j:Q(a)_j \neq 0 \mbox{ for some } a>0\}.
 \]
 Then $Q(a)_j=0$ for all $j>i$ and $a >0.$
 %Then there exists $0 < a \leq s$ such that $Q(a)_i \neq 0$, 
  Hence $\HF(G(I))_{u-k} = H(0)_{u-k}$ for all $u-k>i.$ Since the Hilbert functions of $Q(0)$ and $G(I)$ are symmetric,
  \[
  \HF(G(I))_k=\HF(G(I))_{u-k}=H(0)_{u-k}=H(0)_k \mbox{ for all } k<u-i.
  \]
 In particular, for $k = u-a-i < u-i$ we obtain 
 $H(a)_{u-a-i} = 0$ for all $a>0.$ Hence by Proposition \ref{Prop:technical}, 
$H(a)_{i}= H(a)_{u-a-i} = 0$ for all $a>0,$ which is a contradiction.

\end{proof}

The above result helps us to conclude the Gorensteinness of the associated graded ring by computing the Hilbert function alone. 

%%%%%%%%%%%%
 
\section{Some open problems in codimension two}

In this section we assume that $A=\dfrac{k[\![ x,y ]\!]}{J}$ is an AG $k$-algebra. In the codimension two case, Iarrobino proved that the symmetric decomposition of the Hilbert function of $G(\m)$ is unique.
We show that for any ideals in $A$ such a result is not necessarily true.

Using Theorem \ref{Thm:BN96} we have written the code ``isHFA" which is useful to understand the $b$-admissible sequences of standard graded rings over an Artinian ring.
This code is written using the similar ideas of Macaulay2 function ``isHF" used to verify admissible Hilbert functions of standard graded rings over the field of rationals. But unlike in the code ``isHF", in the code ``isHFA" we need to provide an integer $b$ and a numerical sequence $h$ in the input.  The output is ``true" if $h$ is the Hilbert function of $\frac{R_0[x_1,\ldots,x_b]}{I}$ for some ideal $I$ in $R_0[x_1,\ldots,x_b]$ where $R_0$ is Artinian, else is ``false".  We use this code to verify that a numerical sequence is $b$-admissible in this paper.
The code is as follows:
\begin{verbatim}
    loadPackage "LexIdeals"
    isHFA = method(TypicalValue=>Boolean)
    isHFA(ZZ, List) := (b,hilb) ->
            ( result:=true;
            if not all(hilb,i->instance(i,ZZ)) then result=false
            else( if hilb#1 > b*hilb#0 then result=false
            else (
            leng:=#hilb;
            degr:=1; while result==true and degr < leng-1
            do ((q,r) := 
                quotientRemainder(hilb#(degr), binomial(degr+b-1, b-1)); 
            if hilb#(degr+1) > q*(binomial(degr+b, b-1))+macaulayBound(r,degr) 
                then result=false else degr=degr+1; );
                );
                    );
            result
            )

--Example:
i1: isHFA(2, {2,2,3,4})
o1: true

\end{verbatim}

Let $A=\dfrac{k[\![ x,y ]\!]}{J}$ be an AG $k$-algebra with the maximal ideal $\m$. Inspired by the results from the symmetric decomposition of the Hilbert function of $G(\m)$ in \cite{I89} we ask:
\begin{question}
\begin{enumerate}[i.]\label{Question}
    \item Which $b$-admissible sequences are $I$-Gorenstein sequences for an ideal $I$ in $A?$
    \item Does every $I$-Gorenstein sequence have a unique symmetric decomposition? 
\end{enumerate}
\end{question}

We remark that a characterization of Gorenstein sequence of codimension two is classically known due to F.S. Macaulay \cite{Macaulay1904} \footnote{In \cite{Macaulay1904} Macaulay concluded a proof using his inverse systems and some ideas of C. A. Scott in \cite{Scott1902}. The characterization of local Gorenstein sequences was also obtained by Brian\c con \cite{Br77}, and later Iarrobino in \cite[Chapter 2]{I94Memoir} presented a
self-contained proof using symmetric decomposition.}. Namely, an admissible sequences of the form 
%Artinian Gorenstein standard algebra of codimension two are of the form 
\[H(A) = (1,2, \ldots, d, h_d, h_{d+1}, \ldots, h_s=1),\]
where $d \geq h_d \geq h_{d+1} \geq h_s =1,$ is a Gorenstein sequence if and only if $\mid h_i-h_{i+1}\mid \leq 1$ for all $i.$
\vskip 2mm

Inspired by this result, we guessed that if a $b$-admissible sequence $(h_0,\ldots,h_s)$ is $I$-Gorenstein for an ideal $I$ in $A,$ then $\mid h_i-h_{i+1} \mid \leq h_0.$ The following example shows that this need not to be true.

\begin{example} 
    Consider a $3$-admissible sequence $h=(3,7,7,3).$ Then $h$ is an $I$-Gorenstein sequence. Indeed,  
    let $A= \frac{k [\![ x,y ]\!]}{(x^4+y^4, x^3y^2)}$ and  $I = (x^2, y^2, xy) \subset A$ be an ideal in $A$. Then $A$ is an AG $k$-algebra, and $G(I)$ has the Hilbert function $h.$ 
    
    Note that here $\mid h_1-h_0 \mid=4>3,$ and $h$ is not a 2-admissible sequence. 
    
\end{example}

 We conjecture the following for a 2-admissible sequence to be $I$-Gorenstein. 
We say that a sequence $H=(h_0, h_1 , \ldots, h_u)$ is {\it partially strict unimodal} if $h_0<h_1< \cdots < h_i \geq h_{i+1} \geq \cdots \geq h_u  $ for some $0 \leq i \leq u.$

\begin{conjecture}
\label{Conj:2admissible-IGoresteinsSeq}
      Let $A = k [\![ x,y ]\!] /J$ be an AG  $k$-algebra. Let $H=(h_0, h_1 , \ldots, h_u)$  be a $2$-admissible sequence 
      such that\\
      (i) $H$ is partially strict unimodal; and\\
      (ii)  $\mid h_i - h_{i+1} \mid \leq h_0$ for all $i$.\\
  Then $H$ is an $I$-Gorenstein sequence. 
\end{conjecture}

A reason for assuming $H$ is partially strict unimodal is that if $H$ is an O-sequence of the form $(1,2,\ldots,h_s),$ then by Macaulay's condition $H$ is partially strict unimodal, which need not be true if $H$ is a $b$-admissible sequence. In the following we give an example of a $2$-admissible sequence, which shows that the assumption (i) in Conjecture \ref{Conj:2admissible-IGoresteinsSeq} is necessary.

\begin{example}
    Consider a 2-admissible sequence $H=(2,2,3,2)$. Then $H$ is not partially strict unimodal, but $\mid h_i - h_{i+1} \mid \leq h_0=2$ for all $i.$
    But $H$ is not an $I$-Gorenstein sequence because $H$ does not admit a symmetric decomposition. 
\end{example}

In Section 5, we investigate $2$-admissible sequences with $h_0 = 2$ and $u \leq 3$, which supports our conjecture.

\vskip 2mm

Now we investigate Question \ref{Question} (ii). 
In \cite{I89} Iarrobino showed that if $A$ is an 
AG $k$-algebra of codimension two, then the Hilbert function of $A$ determines the symmetric decomposition. Moreover, each $Q(a)$ is isomorphic, after a shift in grading, to a complete intersection $k[x,y]/(f_a,g_a)$ \cite[Theorem 2.2]{I94Memoir}. The example below shows that a $2$-admissible $I$-Gorenstein sequence need not determine the symmetric decomposition. 

 \begin{example}
 \label{Example:2SymmDec}
     {\bf $I$-Gorenstein sequence with two different symmetric decompositions}: 
     Consider a $2$-admissible sequence $H=(4,4,1).$ 
     Let $A=Q[\![x,y]\!]/(x^2y^2-x^3, y^3)$, and $I_1=(x^2,y^2)$ an ideal in $A$. Then the Hilbert function of $G(I_1)$ is $H$. The symmetric decomposition of the Hilbert function of $G(I_1)$ is given in Figure 1.
     Further, consider the ring $B = Q[\![x,y]\!]/(y^4-x^5, xy),$ and the ideal $I_2 = (x^3,y^2)$ in $B$. Then the Hilbert function of $G(I_2)$ is also $H$. The Hilbert function of $G(I_2)$ has the symmetric decomposition shown in Figure 2.
     
 \end{example}

\begin{center}
\begin{minipage}{0.45\textwidth}
\centering
{\setlength{\tabcolsep}{1pt}
\begin{tabular}{c | c c c c c }
        \hline
        $\HF(G(I_1))$ \, &&& 4, &4, &1 \\
        \hline
      H(0) \,&&& 1, &2, &1  \\
      H(1) \,&&& 2, & 2& \\
      H(2) \, &&& 1
\end{tabular}\captionof{figure}{}}

\end{minipage}%
\hspace{-2em}% Adds horizontal space between the tables
\begin{minipage}{0.45\textwidth}
\centering
{\setlength{\tabcolsep}{1pt}
\begin{tabular}{c | c c c c c }
        \hline
        $\HF(G(I_2))$ \, &&& 4, &4, &1 \\
        \hline
      H(0) \,&&& 1, &1, &1  \\
      H(1) \,&&& 3, & 3& \\
\end{tabular}\captionof{figure}{}}
\end{minipage}
\end{center}
 
We conjecture that a $2$-admissible $I$-Gorenstein sequence with $h_0=2$ determines the symmetric decomposition.

\begin{conjecture}
    Let $H$ be a $2$-admissible $G(I)$-Gorenstein sequence where $I$ is an ideal in an AG $k$-algebra $A=k[\![x,y]\!]/J$. Then the symmetric decomposition of $H$ is unique.
\end{conjecture}

\section{$2$-admissible $I$-Gorenstein sequences with end degree at most $3$ and $h_0=2$}
In this section, we give a complete list of $2$-admissible $I$-Gorenstein sequences with $h_0=2$ and of end degree at most $3$ where $I$ is an ideal in $A=k[\![x,y]\!]/J$. This list shows that the Conjecture \ref{Conj:2admissible-IGoresteinsSeq} is true if $H$ is an $I$-admissible sequence with $h_0=2$, and end degree at most $3.$
 
Let $I$ be an ideal in an AG $k$-algebra $A=k[\![ x,y ]\!]/J$ such that $G(I)$ has the Hilbert function $H$ with $h_0=2.$ Then $\ell(A/I) = 2.$ The following lemma shows that in this case we can assume that $I=(x,y^2)$, up to a change of coordinates.

\begin{lemma}\label{Lem:ChangeofCoordinatesLength2}
    Let $A=k[\![ x,y ]\!]/J$ be an AG $k$-algebra with an embedding dimension two. If $I$ is an ideal in $A$ such that $\ell(A/I) = 2,$ then $I=(x,y^2)$ up to a change of coordinates.
\end{lemma}
\begin{proof}
Suppose  $\ell(A/I) = 2$. Let $\m=(x,y)A.$
Then  we have a composition series $I \subset \m \subset A$ with $m/I \cong A/\m.$
 % $\frac{\m}{I} \cong \frac{A}{\m}.$ 
 Hence we obtain $\m^2 \subseteq I.$ Since $\ell(A/I)=2$ whereas $\ell(A/\m^2)=3,$ $\m^2 \subsetneq I.$ We choose an element $v \in I\setminus \m^2.$ Let $\frac{\m}{I} = \langle u \rangle $ as a $k$-vector space.
\noindent\newline {\underline{\bf Claim}:  $\m=(u,v).$} 

By Nakayama's lemma it is enough to show that $\bar u,\bar v$ are $k$-linearly independent in ${\m}/{\m^2}$ where $\bar{u},\bar{v}$ denote the image of $u,v$ in $\m/\m^2.$
First note that $u,v \in \m \setminus \m^2.$
Suppose, $u, v$ are linearly dependent in ${\m}/{\m^2}.$ Then we may assume that  $\bar{u} = \bar{c} \bar{v}$ for some $c \in A.$ Hence $ u -cv \in \m^2 \subsetneq I.$ Since $v \in I,$ we obtain that $u \in I$. This is a contradiction because $\bar{u}$ is a nonzero element in $\m/I$. Hence $u, v$ are linearly independent in  ${\m}/{\m^2}.$

Therefore $\ell(A/(v,u^2))=2.$ 
Since we have $(v, u^2) \subset I$ and $\ell(A/(v,u^2)) = \ell(A/I) =2,$ we get $I=(v, u^2)$ Therefore by a change of coordinates, we can assume that $I=(x,y^2).$

\end{proof}

\begin{proposition}\label{Prop:2andLast11}
    Let $\h =(2,h_1, h_2, \ldots, h_{u-2},h_{u-1},1)$ be a $b$-admissible sequence. If $\h$ is an $I$-Gorenstein sequence for an ideal $I$ in an AG $k$-algebra $A =k[\![x,y]\!]/J$ with $\h(1)_0 = 0$, and $\h(0) = (1,1, \ldots, 1)$, then $\h(1)_1 \leq 1.$
\end{proposition}
\begin{proof} By definition $\h(a)_j=0$ whenever $a+j \geq u+1.$
 Therefore  by Theorem \ref{Thm:SymmetricDecomposition},  we have  
    \[\h(1)_{u-1} +\h(0)_{u-1} = h_{u-1}.\]
    Since $H(1)_0 = 0,$ by Proposition \ref{Prop:technical} we have $\h(1)_{u-1} =\h(1)_0 =  0.$

Therefore, the decomposition table looks like the following:
\begin{center}
    {\setlength{\tabcolsep}{1pt}
\begin{tabular}{c | c c c c c c c}
\hline 
    H &&& 2, & $h_1$, & $\cdots$ & $h_{u-1}$, & $h_u$ \\
\hline
      H(0) \,&&& 1,&1, & $\cdots$ & 1,& 1  \\
      H(1) \,&&& 0 & - &  $\cdots$ &  0 & \\
      H(2) \,&&& - & - & $\cdots$ & & \\
      \vdots &&& &  &          & &\\
      \hline
      
  \end{tabular}} 
\end{center}
By definition $Q(0)_0 = \frac{A}{(0:I^u)}.$ Since $\h(0)_0 = 1,$ we get 
\begin{equation}\label{eq:m=colonIdeal}
    \m = (0:I^u).
\end{equation}
Moreover, by Equation \eqref{EQ:Q} $Q(1)_0 = \frac{(0:I^u)}{I \cap (0:I^u) + (0:I^{u-1})}.$ Hence $H(1)_0 = 0$ implies that
\begin{equation}\label{EQ:5.2}
    (0:I^{u}) = I \cap (0:I^u) + (0:I^{u-1}).
\end{equation}
Substituting Equation \eqref{eq:m=colonIdeal} in Equation \eqref{EQ:5.2}, we obtain
\begin{equation}\label{EQ:maximalIdealEquals}
    \m = I + (0:I^{u-1}).
\end{equation}

\noindent Further, by Equation \eqref{EQ:Q}
\[Q(0)_1 = \frac{I \cap (0:I^u)}{I^2\cap (0:I^u) + I \cap (0:I^{u-1})} = \frac{I}{I \cap (0:I^{u-1})}.\]
Since $\ell(A/I) = 2,$ and $\h(0)_1 = 1$, we obtain 
\begin{equation}\label{EQ:len}
 \ell\left(\frac{A}{I \cap (0:I^{u-1})}\right) = 3.
\end{equation}
\noindent  By  Equation \eqref{EQ:maximalIdealEquals} we have
\begin{equation*}
    \begin{split}
        \dfrac{\m}{I} & =\dfrac{I + (0:I^{u-1})}{I}\\
                    & = \dfrac{(0:I^{u-1})}{I \cap (0:I^{u-1})}.
    \end{split}
\end{equation*}
Since $\ell\left(\dfrac{\m}{I}\right) = 1,$ we obtain $\ell\left(\dfrac{(0:I^{u-1})}{I \cap (0:I^{u-1})}\right) = 1.$ 
By using Equation \eqref{EQ:len}, we obtain 
$\ell\left(\dfrac{A}{(0:I^{u-1})}\right) = 2.$
Hence by Lemma \ref{Lem:ChangeofCoordinatesLength2}, after a change of coordinates, we can assume that  $(0:I^{u-1}) = (z,w^2)$ where $\m = (z,w)$ in $A$. 

Since $\ell(A/I)=2,$ by Lemma \ref{Lem:ChangeofCoordinatesLength2} we may assume that $I=(x,y^2),$ up to a change of coordinates.

\noindent {\bf \underline{Claim}}: $\bar z, \bar x$ are linearly independent in $\frac{\m}{\m^2}$, and hence $\m=(x,z).$

Suppose not. Then $z -\alpha x \in \m^2$ for some $\alpha \in A$. 
Note that as $\ell(\m/I)=1,$ $\m^2 \subseteq I=(x,y^2).$ Hence we obtain $z \in I,$ which gives that
\[
(0:I^{u-1}) =(z,w^2)\subseteq I.\] By Equation \eqref{EQ:maximalIdealEquals} we obtain \[\m = I+(0:I^{u-1}) = I,\] which is a contradiction. 
This proves the claim. 

Therefore $\ell(A/(x,z^2))=2.$ As $(x,z^2) \subseteq I$ and $\ell(A/I)=2$, 
 we obtain $I = (x,z^2)$.
Since $z \in (0:I^{u-1}),$ we get $xz, z^3 \in  I \cap (0:I^{u-2}).$ Therefore the set $\{1,x,z,z^2\}$ spans the $k$-vector space $\frac{A}{I^2 +I \cap (0:I^{u-2})},$ and hence
\begin{equation}\label{EQ:Qhilb}
    \ell\left(\frac{A}{I^2 +I \cap (0:I^{u-2})}\right) \leq 4.
\end{equation}
By Equation \eqref{EQ:Q} we have \[Q(1)_1 = \frac{I \cap (0 : I^{u-1})}{I^2 \cap (0:I^{u-1})+ I \cap (0:I^{u-2})} = \frac{I \cap (0 : I^{u-1})}{I^2 + I \cap (0:I^{u-2})}. \]
Using Equation  \eqref{EQ:Qhilb} and \eqref{EQ:len} we obtain 
\[\h(1)_1 = \ell \left(\frac{A}{ I^2 + I \cap (0:I^{u-2})} \right) - \ell\left(\frac{ A}{ I \cap (0:I^{u-1})}\right) \leq 4 - 3 = 1.\]

\end{proof}

%%%%%%%%%%%

\begin{remark}
 Proposition \ref{Prop:2andLast11} is useful in obtaining information about the symmetric decomposition of the Hilbert function of $G(I)$ in certain cases. For example, if $\h=(2,4,3,1,1)$ is an $I$-Gorenstein sequence, then  by Proposition \ref{Prop:2andLast11}, we conclude that $\h$ has the unique symmetric decomposition %is uniquely defined by the Hilbert function as 
 given below:
\begin{center}
    {\setlength{\tabcolsep}{1pt}
\begin{tabular}{c | c c c c c c c}
\hline 
    H &&& 2, & 4, & 3, & 1, & 1 \\
\hline
      H(0) \,&&& 1,& 1, & 1, & 1,& 1  \\
      H(1) \,&&& 0 & 1 &  1 &  0 &  \\
      H(2) \,&&& 1 & 2 & 1 & & \\
      H(3) \,&&& 0 & 0 &           & &\\
      H(4) \,&&& 0 &  & & & \\
      \hline
      
  \end{tabular}} 
\end{center}

\end{remark}
\vskip 2mm
\noindent We are now ready to describe the complete set of $2$-admissible $I$-Gorenstein sequences for an ideal $I$ in $A=k[\!| x,y |\!]/J$ with the end degree at most $3$ and $h_0=2$. If $\h = (2, h_1, \ldots, h_u)$ is an $I$-Gorenstein sequence, then by Corollary \ref{Cor:HuAtmostH0} we get $h_u \leq 2.$
Below is the list of all $2$-admissible sequences up to $u =3$ with $h_u \leq 2$, where we indicate the $I$-Gorenstein sequences in bold. 

\begin{table}[ht]
    \centering
    \begin{tabular}{|c|c|c|c|c|c|c|}
         \hline
         \bf{(2,1)} & \bf{(2,2)} & \bf{(2,1,1)} & \bf{(2,2,1)} &
         \bf{(2,2,2)}  & \bf{(2,3,1)} & \bf{(2,3,2)} \\
         \hline {(2,4,1)} & \bf{(2,4,2)} & \bf{(2,1,1,1)} & \bf{(2,2,1,1)} & \bf{(2,2,2,1)} & \bf{(2,2,2,2)} & (2,2,3,1)  \\
         \hline
         (2,2,3,2)  & \bf{(2,3,1,1)} & \bf{(2,3,2,1)} & \bf{(2,3,2,2)}& \bf{(2,3,3,1)} & \bf{(2,3,3,2)} &(2,3,4,1) \\
         \hline (2,3,4,2) &  (2,4,1,1) & 
         \bf{(2,4,2,1)} & \bf{(2,4,2,2)} & \bf{(2,4,3,1)} & 
        \bf{(2,4,3,2)} & (2,4,4,1) \\
        \hline
        \bf{(2,4,4,2)} & \multicolumn{6}{c}{} \\
         \cline{1-1}
\end{tabular}
\vspace{0.25em}
    \caption{2-admissible sequences up to u=3 with $h_u \leq 2$.}
    \label{tab:tab_1}
\end{table}

\vspace{.4em}

In Table \ref{table:IGorenstein} we justify that the sequences indicated in bold in Table \ref{tab:tab_1} are $I$-Gorenstein. In fact, for each such sequence $\h$ we provide a polynomial $F$ in the divided power ring $k_{DP}[X,Y]$ so that the apolar algebra $A_F:=k[\![x,y]\!]/\ann_{k[\![x,y]\!]}(F)$ is an AG $k$-algebra (see \cite[Appendix A]{IK99}, \cite[Appendix A2.4]{Ei95} for basic properties of divided power rings, and \cite{ER17} for definition of apolar algebra). Then $I=(x,y^2)$ is an ideal in $A_F$ such that $G(I)$ has the Hilbert function $\h.$ Moreover, in Table \ref{table:IGorenstein} we also describe the symmetric decomposition of the Hilbert function of $G(I)$.

\begin{table}[ht!]
{\centering
\begin{tabular}{|c|c|c| |c|c|c|}
         \hline
          {\bf  $I$-Gorenstein }  &{ \bf Form} & {\bf symmetric } \, &  {\bf  $I$-Gorenstein }  &{ \bf Form} & {\bf symmetric } \,
           \\{\bf sequence HF} & {\bf  $F$} & {\bf decomposition} & {\bf sequence HF} & {\bf  $F$} & {\bf decomposition}\\
            \hline 
$(2,1)$ & $X+Y^2$ & {\setlength{\tabcolsep}{1pt}
\begin{tabular}{c  c c c c }
        HF  \, &&& 2, &1 \\
        \hline
      H(0) \,&&& 1, &1  \\
      H(1) \,&&& 1 &\\
  \end{tabular}} & 
            $(2,2)$ & $XY$ & {\setlength{\tabcolsep}{1pt}
\begin{tabular}{c  c c c c }
        HF \,&&& 2, & 2 \\
        \hline
      H(0) \,&&& 2, & 2  \\
      H(1) \,&&& 0 &\\
  \end{tabular}}\\
            \hline
            $(2,1,1)$ & $X^2+Y^2$ & {\setlength{\tabcolsep}{1pt}
\begin{tabular}{c  c c c c c }
        HF \, &&& 2, &1, &1 \\
        \hline
      H(0) \,&&& 1, &1, &1  \\
      H(1) \,&&& 1, & 0& \\
  \end{tabular}}  &
            $(2,2,1)$ & $X^2+Y^3$ &  {\setlength{\tabcolsep}{1pt}
\begin{tabular}{c  c c c c c }
        HF \, &&& 2, &2, & 1\\
        \hline
      H(0) \,&&& 1, &1, &1  \\
      H(1) \,&&& 1, & 1& \\
      H(2) \,&&& 0
  \end{tabular}}  \\
            \hline
            $(2,2,2)$ & $X^2Y$ & {\setlength{\tabcolsep}{1pt}
\begin{tabular}{c  c c c c c }
        HF \,&&& 2, &2, &2\\
        \hline
      H(0) \, &&& 2, &2, &2 
\end{tabular}} &
            $(2,3,1)$ & $XY^2$ & {\setlength{\tabcolsep}{1pt}
\begin{tabular}{c c c c c c }
        HF \, &&& 2, &3, &1 \\
        \hline
      H(0) \,&&& 1, &2, &1  \\
      H(1) \,&&& 1, & 1& \\
      H(2) \,&&& 0
  \end{tabular}} \\
            \hline 
            $(2,3,2)$ & $X^2Y+Y^4$ & {\setlength{\tabcolsep}{1pt}
\begin{tabular}{c c c c c c }
        HF \, &&& 2, &3, &2\\
        \hline
      H(0) \,&&& 2, &3, &2  \\
      
  \end{tabular}} &
            $(2,4,2)$ & $XY^3$ & {\setlength{\tabcolsep}{1pt}
\begin{tabular}{c c c c c c }
        HF \,&&& 2, & 4, & 2\\
        \hline
      H(0) \,&&& 2, &4, &2  \\
      
  \end{tabular}}  \\
            \hline
     $(2,1,1,1)$ & $X^3+Y^2$ &{\setlength{\tabcolsep}{1pt}
\begin{tabular}{c  c c c c c c }
        HF \, &&& 2, &1, &1, &1\\
        \hline
      H(0) \,&&& 1, &1, &1, &1  \\
      H(3) \,&&& 1, & 0& \\
  \end{tabular}} &
            $(2,2,1,1)$ & $X^3+Y^3$ & {\setlength{\tabcolsep}{1pt}
\begin{tabular}{c  c c c c c c}
        HF \, &&& 2, &2, &1, &1 \\
        \hline
      H(0) \,&&& 1, &1, &1, & 1  \\
      H(2) \,&&& 1, & 1& &\\
  \end{tabular}}  \\
            \hline
            $(2,2,2,1)$ & $X+Y^6$ & {\setlength{\tabcolsep}{1pt}
\begin{tabular}{c  c c c c c c}
        HF \, &&& 2, &2, &2, &1 \\
        \hline
      H(0) \,&&& 1, &1, &1, &1  \\
      H(1) \,&&& 1, & 1,&1 \\
  \end{tabular}} &
            $(2,2,2,2)$ & $X^3Y$ & {\setlength{\tabcolsep}{1pt}
\begin{tabular}{c  c c c c c c}
        HF \, &&& 2, &2, &2,&2 \\
        \hline
      H(0) \,&&& 2, &2, &2, &2  \\
  \end{tabular}} \\
            \hline
            $(2,3,1,1)$ & $X^3+Y^4$ & {\setlength{\tabcolsep}{1pt}
\begin{tabular}{c  c c c c c c}
        HF \, &&& 2, &3, &1, &1 \\
        \hline
      H(0) \, &&& 1, &1, &1, &1\\
      H(1) \,&&& 0, &1, &0, &1  \\
      H(2) \,&&& 1, & 1& \\
  \end{tabular}} &
            $(2,3,2,1)$ & $X^2+Y^6$ & {\setlength{\tabcolsep}{1pt}
\begin{tabular}{c  c c c c c c }
        HF \, &&& 2, &3, &2, &1 \\
        \hline
      H(0) \,&&& 1, &1, &1, &1 \\
      H(1) \,&&& 1, & 2, & 1 \\
  \end{tabular}} \\
            \hline
            $(2,3,2,2)$ & $ X^2+Y^7$ & {\setlength{\tabcolsep}{1pt}
\begin{tabular}{c  c c c c c c}
        HF \, &&& 2, &3, &2, &2 \\
        \hline
      H(0) \,&&& 2, &2, &2, &2  \\
      H(1) \,&&& 0, & 1,&0 \\
  \end{tabular}} &
            $(2,3,3,1)$ & $X^2Y^2$ & 
            {\setlength{\tabcolsep}{1pt}
\begin{tabular}{c  c c c c c c }
        HF \, &&& 2, &3, &3, &1 \\
        \hline
      H(0) \,&&& 1, &2, &2, &1  \\
      H(1) \,&&& 1, & 1,& 1 \\
  \end{tabular}} \\
            \hline
            $(2,3,3,2)$ & $X^3+Y^7$ &{\setlength{\tabcolsep}{1pt}
\begin{tabular}{c  c c c c c c }
        HF \, &&& 2, &3, &3, &2\\
        \hline
      H(0) \,&&& 2, &3, &3, &2 \\     
  \end{tabular}}  &
             $(2,4,2,1)$ & $X^2Y+Y^6$ & {\setlength{\tabcolsep}{1pt}
\begin{tabular}{c  c c c c c c }
        HF \, &&& 2, &4, &2, &1 \\
        \hline
      H(0) \,&&& 1, &1, &1, &1  \\
      H(1) \,&&& 1, & 3,& 1 \\
  \end{tabular}}\\
            \hline 
            $(2,4,2,2)$ & $X^2Y+Y^7$ & {\setlength{\tabcolsep}{1pt}
\begin{tabular}{c  c c c c c c }
        HF \, &&& 2, &4, &2, &2 \\
        \hline
      H(0) \,&&& 2, &2, &2, &2  \\
      H(1) \,&&& 0, & 2,& 0 \\
  \end{tabular}}&
            $(2,4,3,1)$ & $XY^4$ &  {\setlength{\tabcolsep}{1pt}
\begin{tabular}{c  c c c c c c }
        HF \, &&& 2, &4, &3, &1 \\
        \hline
      H(0) \,&&& 1, &2, &2, &1  \\
      H(1) \,&&& 1, & 2,& 1 \\
  \end{tabular}}\\
            \hline
            $(2,4,3,2)$ &$Y^7+X^2Y^2$ &{\setlength{\tabcolsep}{1pt}
\begin{tabular}{c  c c c c c c }
        HF \, &&& 2, &4, &3, &2 \\
        \hline
      H(0) \,&&& 2, &3, &3, &2  \\
      H(1) \,&&& 0, & 1,& 0 \\
  \end{tabular}} &
            $(2,4,4,2)$ &  $XY^5$ & {\setlength{\tabcolsep}{1pt}
\begin{tabular}{c  c c c c c c }
        HF \, &&& 2, &4, &4, &2 \\
        \hline
      H(0) \,&&& 2, &4, &4, &2  \\
  \end{tabular}}\\
            \hline  
        
    \end{tabular} \\
}
\vspace{0.2em}
 \caption{$2$-admissible $I$-Gorenstein sequences with a symmetric decomposition}
    \label{table:IGorenstein}
\end{table}
%%%%%%%%%%%%%%%%%%%%%%%%%%%%%%%%%%%%
\vspace{1em}

 Next we prove that each of the sequences that are not bold 
 %the $2$-admissible sequences which are not 
 in the Table \ref{tab:tab_1} are not $I$-Gorenstein sequences. Since $\ell(A/I) = 2$, by Lemma \ref{Lem:ChangeofCoordinatesLength2} we assume $I=(x,y^2).$

\noindent {\bf (2,4,1)}: Suppose (2,4,1) is an $I$-Gorenstein sequence. Then the possible symmetric decompositions are
\[
{\setlength{\tabcolsep}{1pt}
\begin{tabular}{c | c c c c c}
\hline
      H \,& & &  2, & 4, & 1\\
\hline
      H(0) \,&& & 1, &3, &1  \\
      H(1) \,&&& 1, &1 & \\
      H(2) \,&&& 0 & & \\
     \hline
\end{tabular}}
\hspace{3em}
{\setlength{\tabcolsep}{1pt}
\begin{tabular}{c | c c c c c}
\hline
      H \,& & &  2, & 4, & 1\\
\hline
      H(0) \,&& & 1, &4, &1  \\
      H(1) \,&&& 0, &0 & \\
      H(2) \,&&& 1 & & \\
     \hline
\end{tabular}}
\]
Since $H(0)_0 = 1$ in both the symmetric decompositions, by definition of $Q(0)_0$ we obtain $(0: I^2) = \m.$  Hence $\m I \subseteq (0:I)$ which gives that  $xy, y^3 \in (0:I).$ Therefore $\ell\left(A/I^2+I \cap (0:I)\right) \leq 4.$  Since by Equation \eqref{EQ:Q} we have $Q(0)_1 = \frac{I}{I^2+I \cap (0:I)}$, we obtain $\h(0)_1 = \ell(A/I^2+I \cap (0:I)) - \ell(A/I) \leq 4- 2 = 2$. This is a contradiction in both the symmetric decompositions. Hence $(2,4,1)$ is not an $I$-Gorenstein sequence.
\vskip 2mm

\noindent {\bf (2,2,3,1)}:
If $H=(2,2,3,1)$ is an $I$-Gorenstein sequence, then 
$\ell(A/I) = 2,$ $\ell(A/I^2) = 4,$ $\ell(A/I^3) =7,$ and the only possible symmetric decomposition of $H$ is
\[{\setlength{\tabcolsep}{1pt}
\begin{tabular}{c | c c c c c c}
\hline
      H \,&&&  2, & 2, & 3, &1\\
\hline
      H(0) \,&&& 1, &2, &2, & 1 \\
      H(1) \,&&& 1, & 0, & 1 & \\
      H(2) \,&&& 0, & 0& \\
      H(3) \,&&& 0 &  & & \\
      \hline
     
\end{tabular}} 
\]
Since $Q(0)_0 = A/(0:I^3)$ and $\h(0)_0 = 1,$ we get $(0:I^3) = \m.$ Hence $\m I \subseteq (0:I^2).$ Therefore we obtain
\begin{equation}\label{EQ:0colonIcube=m}
    xy, \, y^3 \in  I \cap (0:I^2).
\end{equation}
By Equation \eqref{EQ:Q} we have 
\begin{equation*}
    Q(1)_0 = \frac{I^0 \cap (0 : I^3)}{I \cap (0:I^3) + I^0 \cap (0:I^2)} = \frac{(0:I^3)}{I + (0:I^2)}.
\end{equation*}
Since $\ell\left(A/(0:I^3)\right)=1$ and $\h(1)_0=1,$ we get $\ell(A/I+(0:I^2)) = 2.$ Further, $I \subseteq I + (0:I^2)$ and $\ell(A/I) = 2$ implies that $ I+(0:I^2)=I$, which gives $(0:I^2) \subseteq I.$

Again by Equation \eqref{EQ:Q} we have 
\begin{equation*}
\begin{split}
     Q(0)_1 & = \frac{I \cap (0:I^3)}{I^2 \cap (0:I^3)+I \cap (0:I^2)} \\
     & = \frac{I}{(0:I^2)} \quad \quad \quad ( \mbox{since } (0:I^2) \subseteq I).
\end{split}
\end{equation*}
Since $\h(0)_1 = 2$ and $\ell(A/I) =2$, we get $ \ell\left(A/(0:I^2)\right) = 4.$
Moreover $I^2 \subseteq (0:I^2) $ and $\ell(A/I^2) = 4,$ we obtain
\begin{equation}\label{EQ:Isquare}
    I^2 =  (0:I^2).
\end{equation}
Hence by Equation \eqref{EQ:0colonIcube=m}, we get
\begin{equation*}\label{xyinI^2}
    xy, y^3 \in I^2.
\end{equation*}
Note that $I^2 = (x^2, xy^2, y^4) +J$  in $k[\![ x,y]\!]/J$. 
Since $\ell\left(k[\![x,y]\!]/J\right) = 8,$ we get $\ell\left(k[\![x,y]\!]/J^{*}\right) = 8$ where $J^*$ denotes the ideal of initial forms of $J$ in $k[\![x,y]\!].$
Since $xy \in I^2 = (x^2, xy^2,y^4)+J$ we obtain 
\begin{equation}\label{Eq:1for J}
xy - c_1x^2 \in J^*    
\end{equation}
for some $c_1 \in k$. Similarly, $y^3 \in I^2,$ implies that
\begin{equation}\label{EQ:2forJ}
    y^3 - (d_1x+d_2y)x^2 +d_3xy^2 \in J^*
\end{equation}
for some $d_1, d_2, d_3 \in k.$
Since $x^4\in J,$ by Equations \eqref{Eq:1for J} and \eqref{EQ:2forJ} we obtain that the set $\{1,x,y,x^2, y^2, x^3\}$ spans the $k$-vector space $k[\![ x,y]\!]/J^*$. Hence $\ell \left( k[\![ x,y]\!]/J^*\right) \leq 6$ which is a contradiction.

\vskip 2mm

\noindent {\bf (2,2,3,2)}: Since $(2,2,3,2)$ does not admit a symmetric decomposition, it is not an $I$-Gorenstein sequence.

\vskip 2mm

\noindent {\bf (2,3,4,1)}:
If $H=(2,3,4,1)$ is an $I$-Gorenstein sequence, then the possible symmetric decomposition of $H$ is
\[
{\setlength{\tabcolsep}{1pt}
\begin{tabular}{c | c c c c c c}
\hline
      H \,&&&  2, & 3, & 4, &1\\
\hline
      H(0) \,&&& 1, &3, &3, & 1 \\
      H(1) \,&&& 1, & 0, & 1 & \\
      \hline
\end{tabular}}
\]
Since $Q(0)_0 = \frac{A}{(0:I^3)}$ and $H(0)_0 = 1$, we get $(0:I^3) = \m = (x,y).$ Hence $\m I \in (0:I^2)$. Therefore, $xy, y^3 \in I \cap (0:I^2).$ This gives that $\ell(A/ I^2 + I \cap (0:I^2)) \leq 4.$

Again, by Equation \eqref{EQ:Q}, \[Q(0)_1 = \frac{I \cap (0:I^3)}{I^2 \cap (0:I^3) + I \cap (0:I^2)} = \frac{I}{I^2 + I \cap (0:I^2)}.\] Since $\ell(A/I) = 2$, and $\ell(A/ I^2 +I \cap (0:I^2)) \leq 4,$ we obtain $H(0)_1 \leq 2.$ This is a contradiction.

\noindent {\bf (2,3,4,2)}:
Since $(2,3,4,2)$ does not admit a symmetric decomposition, it is not an $I$-Gorenstein sequence.
\vskip 2mm

\noindent {\bf (2,4,1,1)}: Suppose $H=(2,4,1,1)$ is an $I$-Gorenstein sequence. Then the possible symmetric decomposition of $H$ is
\[
{\setlength{\tabcolsep}{1pt}
\begin{tabular}{c | c c c c c c}
\hline
      H \,&&&  2, & 4, & 1, &1\\
\hline
      H(0) \,&&& 1, &1, &1, & 1 \\
      H(1) \,&&& 0, & {\footnotesize $\leq$}1, & 0 & \\
      H(2) \,&&& -, & -,& 0\\
      H(3) \,&&& - & 0 & & \\
      \hline
     
\end{tabular}}
\]
Here we conclude $H(1)_1 \leq 1$  by Proposition \ref{Prop:2andLast11}. This shows that $H(2)_1 \geq 2$ which implies that $H(2)_0 \geq 2,$ a contradiction.

\vskip 2mm

\noindent {\bf (2,4,4,1)}: If $H=(2,4,4,1)$ is an $I$-Gorenstein sequence, then the  possible symmetric decompositions of $H$ are 
\[
{\setlength{\tabcolsep}{1pt}
\begin{tabular}{c | c c c c c c}
\hline
      H \, &&&  2, & 4, & 4, &1\\
\hline
      H(0) \,&&& 1, &3, &3, & 1 \\
      H(1) \,&&& 1, & 1, & 1 & \\
      H(2) \,&&& 0, & 0& \\
      H(3) \,&&& 0 &  & &\\
      \hline
     
\end{tabular}} 
\hspace{3em}
{\setlength{\tabcolsep}{1pt}
\begin{tabular}{c | c c c c c c}
\hline
      H \,&&&  2, & 4, & 4, &1\\
\hline
      H(0) \,&&& 1, &4, &4, & 1 \\
      H(1) \,&&& 0, & 0 &  & \\
      H(2) \,&&& 0, & 0 & \\
      H(3) \,&&& 1 &  & & \\
      \hline
     
\end{tabular}}
\]
Since $\h(0)_1 = 1$ in both cases, by using the arguments similar to the case $(2,4,1)$ we get a contradiction.

\bibliographystyle{abbrv}
\bibliography{main}

\end{document}